\documentclass[12pt]{article}
\usepackage[centertags]{amsmath}
\usepackage{amsfonts}
\usepackage{amssymb}
\usepackage{amsthm}
\usepackage{mathrsfs}
\usepackage{newlfont}
\usepackage{graphicx}
\newtheorem{Theorem}{Theorem}[section]

\newtheorem{Remark}{Remark}[section]

\theoremstyle{definition}
\newtheorem{Definition}{Definition}[section]
\theoremstyle{remark}

\begin{document}

\title{Stability and convergence of a higher order rational
difference Equation}
\author{{Hamid Gazor\footnote{ {\rm Corresponding author}\textit{E-mail address}: h.gazor@iaumajlesi.ac.ir } , Saeed Parvandeh\footnote{ \textit{E-mail address}:}}}
\date{}
\maketitle

\begin{center}
\textit{\scriptsize Department of Mathematics, Majlesi Branch, Islamic Azad
University, Isfahan, Iran.}\\[0pt]
\end{center}
\begin{abstract}
In this paper the asymptotic stability of equilibria and periodic
points of the following higher order rational difference Equation
$$x_{n+1}=\frac{\alpha x_{n-k}}{\beta +\gamma  x_{n}x_{n-1} \ldots x_{n-k}}, \  k\geq 1,\ n=0,1,\ldots $$
is studied where the parameters $\alpha ,\beta , \gamma $ are
positive real numbers, and the initial conditions $x_{-k},\ldots
,x_{0}$ are given arbitrary real numbers. The forbidden set of
this equation is found and then, the order reduction method is
used to facilitate the analysis of its asymptotic dynamics.

\end{abstract}

\section{Introduction and preliminaries}

Difference equations may appear as solutions of various phenomena
or as a discretized system of delay or non-delay differential
equations. They are very important in both theory and
applications; for applications in biology (see \cite{BC}),
economics (see \cite{LM,IS}), medical sciences (see \cite{MK}),
military sciences (see \cite{Epstein}). The main goal in the
study of difference equations is to understand the asymptotic
behavior of solutions rather than trying to find an explicit
formula for solutions. This basically is not only because the
explicit solutions are hard to find but also the explicit
solutions may still represent a complex dynamics and they yet may
require a qualitative analysis to understand their dynamics.

The ratio of any two polynomials of a recursive sequence is
called a rational difference equation. Rational difference
equations are important as practical classes of difference
equations. Most of the works in the literature of rational
difference equations have treated the first and second order
difference equations. For second order rational difference
equations we refer the reader to the monograph of Kullenovic and
Ladas (\cite{KL}). In this paper we study the dynamics of the
following \((k+1)\)-order rational difference equation

\begin{eqnarray}\label{E1}
x_{n+1}&=&\frac{\alpha x_{n-k}}{\beta +\gamma  x_{n}x_{n-1} \ldots
x_{n-k}}, \ \ \  n=0,1,\ldots
\end{eqnarray}
where $k\geq 1$ and the parameters $\alpha ,\beta ,\gamma $ are
positive. We allow the initial conditions $x_{-k},x_{-k+1},\ldots
x_{0}$ to take any arbitrary value out of the forbidden set of
equation (\ref{E1}). Let $I$ be some interval of positive real
numbers and $f:I^{k+1}\rightarrow I$ be a continuously
differentiable function. Then for every initial conditions
$(x_{-k},x_{-k+1}, \ldots ,x_{0})\in I^{k+1}$, the difference
equation
\begin{eqnarray}\label{E2}
x_{n+1}=f(x_{n},x_{n-1},\ldots ,x_{n-k}), \ \ \ n=0,1,\ldots
\end{eqnarray}
has a unique solution $\{x_{n}\}_{n=-k}^{\infty }$. The point
$\overline{x}\in I$ is called an equilibrium of
equation~(\ref{E2}) (or simply an equilibrium of $f$) if
$f(\overline{x},\overline{x},\ldots ,\overline{x})=\overline{x}$,
i.e., $x_{n}=\overline{x}$ for all $n\geq 0$ (such a solution is
also called a trivial solution). The point $(c_{1},\ldots
,c_{k+1})$ is called a $(k+1)$-cycle if $x_{(k+1)n-i}=c_{k+1-i}$
for all $i=0,1,\ldots ,k$. In this case we say that $\{x_{n}\}$
is periodic with period $(k+1)$. The linearized equation
associated with equation~(\ref{E2}) about the equilibrium
$\overline{x}$ is
\begin{eqnarray*}
z_{n+1}=\sum _{i=0}^{k}\frac{\partial f}{\partial u_{i}}(\overline{x},\ldots \overline{x})z_{n-i},
\end{eqnarray*}
and its corresponding characteristic equation is defined by
\begin{eqnarray}\label{E3}
\lambda ^{k+1}-\sum _{i=0}^{k}\frac{\partial f}{\partial
u_{i}}(\overline{x},\ldots \overline{x})\lambda ^{k-i}=0.
\end{eqnarray}
An equilibrium $\overline{x}$ of equation(\ref{E2}) is called
locally stable if for every $\epsilon >0,$ there exists $\delta
>0$ such that for the solution $\{x_{n}\}_{n=-k}^{\infty }$ of
equation~(\ref{E2}) with
$|x_{-k}-\overline{x}|+|x_{-k+1}-\overline{x}|+\ldots+|x_{0}-\overline{x}|<\delta
$ we have $|x_{n}-\overline{x}|<\epsilon $ for all $n\geq 1$.
Furthermore, if there exists $\gamma >0$ such that for the
solution $\{x_{n}\}_{n=-k}^{\infty }$ of equation~(\ref{E2}) with
$|x_{-k}-\overline{x}|+|x_{-k+1}-\overline{x}|+\ldots
+|x_{0}-\overline{x}|<\gamma$ we have $\lim _{n\rightarrow \infty
}x_{n}=\overline{x}$, then $\overline{x}$ is called locally
asymptotically stable. Linearized stability theorem indicates
that if all roots of equation~(\ref{E3}) lie inside the open unit
disk $|\lambda |<1$, then the equilibrium $\overline{x}$ of
equation~(\ref{E2}) is locally asymptotically stable. If at least
one of the roots of equation~(\ref{E3}) has modulus greater than
one, then the equilibrium $\overline{x}$ of equation~(\ref{E2}) is
unstable, see {\rm e.g.,}  \cite{KL1}.

$\overline{x}$ is called a global attractor on an interval $I$ if
for every solution $\{x_{n}\}_{n=-k}^{\infty }$ of
equation~(\ref{E2}) with $x_{-k},x_{-k+1},\ldots, x_{0}\in I,$ we
have $\lim _{n\rightarrow \infty }x_{n}=\overline{x}$. When
\(\overline{x}\) is locally stable and a global attractor, we
call it globally asymptotically stable. The forbidden set of
equation~(\ref{E2}) is the set of all (k+1)-tuples $(f_{1},f_{2},
\ldots f_{k+1}),$ where they can not be taken as the initial
conditions for an infinite well-defined sequence \(\{x_n\}\) on
its domain. In other words, the forbidden set of
equation~(\ref{E2}) is the set of all initial conditions such
that arbitrary iterations of the right hand side of
equation~(\ref{E2}) are not well-defined.

The following definitions, lemma, corollary, and theorems are
needed to study the global behavior of solutions of
equation~(\ref{E1}).

\begin{Definition}(for original ideas see \cite{E,S}) \label{D2} Consider equation~(\ref{E2}) and assume that $D\subseteq \Bbb{R}^{k+1}$ is nonempty.
\begin{description}
    \item[\it{(i)}] Suppose that $(x_{1},\ldots x_{k+1})\in \Bbb{R}^{k+1}$. Then $\| x\| =\max \{x_{1},\ldots ,x_{k+1}\}$ denotes the sup-norm
    of $x$. Also, if $\overline{x}$ is the equilibrium of $f$ then
    $\overline{X}=(\overline{x}, \ldots ,\overline{x})\in
    \Bbb{R}^{k+1}$ is an equilibrium of its {\em vectorization} that is
    \begin{equation*}
        V_{f}(x_{1},\ldots x_{k+1})=(f(x_{1},\ldots x_{k+1}), x_{1},\ldots ,x_{k}).
    \end{equation*}
    \item[\it{(ii)}] If there is a non-constant function $H\in
    C(D,\Bbb{R})$ such that $H\circ V_{f}=\phi\circ H$ on $D$ for some $\phi \in
    C(H(D),\Bbb{R})$ then $V_{f}$ is also called a $(D,H,\phi )$-{\em semiconjugate} of $\Bbb{R}^{k+1}$. The mapping $H$ is called a {\em link map} and $\phi $ is called the {\em factor map.} For each $t\in
    H(D)$, the level set $H^{-1}(t)\cap D$, abbreviated as
    $H^{-1}_{t}$ is called a {\em fiber} of $H$ in $D$.
    \item[\it{(iii)}] A continuous mapping $h:D\rightarrow \Bbb{R}$
    is said to be {\em bending} at a point $x\in D$ if $x$ is not an
    isolated point in $D$ and $x\not\in [h^{-1}(h(x))]^{\circ} \cap
    D$, where $S^{\circ}$ denotes the interior of $S$.
    \item[\it{(iv)}] An equilibrium $\overline{x}$ of equation~(\ref{E2})
    is exponentially stable under $f$ relative to some nontrivial
    interval $I$ containing $\overline{x}$ if there is $\gamma
    \in(0,1)$ such that for every solution $\{x_{n}\}$ of
    equation~(\ref{E2}) with initial values $x_{-k},\ldots ,x_{0}\in I$ we
    have for all $n\geq 1$ that $|x_{n}-\overline{x}|<c\gamma ^{n}$ where
    $c=c(x_{0},\ldots x_{-k})$ is independent of $n$.
    \item[\it{(v)}] An equilibrium $\overline{x}$ of a map $f:\Bbb{R}\rightarrow
    \Bbb{R}$ is semistable (from the right) if for any $\epsilon >0$
    there exists $\delta >0$ such that if $0<x_{0}-\overline{x}<\delta
    $ then $|f^{n}(x_{0})-\overline{x}|<\epsilon $ for all $n\geq
    1$. If in addition, $\lim _{n\rightarrow \infty
    }f^{n}(x_{0})=\overline{x}$ whenever $0<x_{0}-\overline{x}<\gamma
    $ for some $\gamma >0$, then $\overline{x}$ is said to be
    semiasymptotically stable (from the right). Semistability (semiasymptotic
    stability) from the left is defined analogously.
\end{description}
\end{Definition}



\begin{Theorem}\label{T2} (see \cite{S}). Let $V_{f}$ is a $(D,H,\phi )$-semiconjugate map and $\overline{x}\in
D$ is an equilibrium of $V_{f}$.
\begin{description}
    \item[\it{(i)}] $\overline{t}=H(\overline{x})$
is an equilibrium of $\phi $.
    \item[\it{(ii)}] (Boundedness). Assume that $|H(x)|\rightarrow \infty $ as $\| x\| \rightarrow \infty $.
    If the sequence $\{\phi ^{n}(t_{0})\}$ is bounded for some $t_{0}\in
    H(D)$, then the sequence $\{V_{f}^{n}(x_{0})\}$ with $x_{0}\in
    H^{-1}_{t_{0}}$ is bounded.
    \item[\it{(iii)}] (Stability and instability). Assume that $H$ is
    bending at $\overline{x}$. If we set
    $\overline{t}=H(\overline{x})$ then
    \begin{description}
        \item[\it{(a)}] If $\overline{x}$ is stable (asymptotically
        stable) under $V_{f}$, then $\overline{t}$ is a stable (asymptotically
        stable) equilibrium of $\phi $.
        \item[\it{(b)}] If $\overline{t}$ is unstable under $\phi $,
        then $\overline{x}$ is unstable under $V_{f}$.
     \end{description}
    \item[\it{(iv)}] (Attractivity of invariant fibers). Let $\overline{t}\in I$ be an isolated equilibrium of $\phi $ which attracts all points in $I$. If $x_{0}\in D\cap H^{-1}(I)$ and
    $\{V_{f}^{n}(x_{0})\}$ is bounded, then $\{V_{f}^{n}(x_{0})\}$ converges to the invariant fiber $H^{-1}(\overline{t})$.
\end{description}
\end{Theorem}

\begin{Theorem}\label{T3}(see \cite{E}). Let $\overline{x}$ be an equilibrium of $f:\Bbb{R}\rightarrow \Bbb{R}$,
$f'(\overline{x})=1$, and $f''(\overline{x})\neq 0$. Then
$\overline{x}$ is semiasymptotically stable from the right (left)
if $f''(\overline{x})<0$ ($f''(\overline{x})>0$).
\end{Theorem}

\section{The Forbidden Set}
Consider equation~(\ref{E1}). If $\alpha =0$, the solution is
trivial. If $\gamma =0$ then equation~(\ref{E1}) reduces to a
linear equation. If $\beta =0$ then the solution is periodic with
period $k+1$. Thus, we assume that all the parameters are nonzero.
A change of variables $x_{n}=\sqrt[k+1]{\frac{\beta}{\alpha
}}y_{n}$ followed by the change $y_{n}=x_{n}$ reduces
equation~(\ref{E1}) to
\begin{eqnarray}\label{E4}
x_{n+1}=\frac{cx_{n-k}}{1+x_{n}x_{n-1}\ldots x_{n-k}}, \ \ \
n=0,1,\ldots
\end{eqnarray}
where $c=\frac{\alpha}{\beta}\sqrt[k+1]{\frac{\beta}{\alpha}}$.
Hence, , we consider equation~(\ref{E4}) instead of
equation~(\ref{E1}), hereafter. Now we are ready to obtain the
forbidden set of equation~(\ref{E4}).

\begin{Theorem}\label{F} Consider equation~(\ref{E4}). Assume that $\mathscr{F}$ is the
forbidden set for this equation. Then
$$\mathscr{F}=\bigcup _{m=0}^{\infty}
\Big\{(f_{1},f_{2},\ldots ,f_{k+1}): f_{1}f_{2}\ldots
f_{k+1}=\frac{-1}{\sum _{i=0}^{m}c^{i}}\Big\}$$

\end{Theorem}
\begin{proof} Assume the initial conditions $x_{-k},
x_{-k+1},\ldots x_{0}$ satisfy
$$x_{-k}x_{-k+1}\ldots x_{0}=\frac{-1}{\sum _{i=0}^{m}c^{i}}$$
for some $m\in \Bbb{N}\cup \{0\}$. Then equation~(\ref{E4})
implies that
$$x_{1}x_{0}\ldots x_{1-k}=\frac{cx_{0}x_{-1}\ldots x_{1-k}x_{-k}}{1+x_{0}\ldots x_{-k}}=\frac{c\left(\frac{-1}{\sum _{i=0}^{m}c^{i}}\right)}{1-\sum _{i=0}^{m}c^{i}}=\frac{-1}{\sum _{i=0}^{m-1}c^{i}}.$$
Therefore, we obtain from equation~(\ref{E4}) that
$$x_{2}x_{1}\ldots x_{2-k}=\frac{cx_{1}x_{0}\ldots x_{2-k}x_{1-k}}{1+x_{1}\ldots x_{1-k}}=\frac{c\left(\frac{-1}{\sum _{i=0}^{m-1}c^{i}}\right)}{1-\sum _{i=0}^{m-1}c^{i}}=\frac{-1}{\sum _{i=0}^{m-2}c^{i}}.$$
Continuing in this fashion by induction we obtain that
$$x_{m}x_{m-1}\ldots x_{m-k}=\frac{-1}{\sum _{i=0}^{m-m}c^{i}}=-1.$$
As a result the iteration process stops at $x_{m+1}$. Conversely,
assume iteration process stops at some point $x_{n_{0}+1}$, i.e.,
$x_{n_{0}}x_{n_{0}-1}\ldots x_{n_{0}-k}=-1$. Then
equation~(\ref{E4}) implies that
$$x_{n_{0}-1}x_{n_{0}-2}\ldots x_{n_{0}-k-1}=\frac{-1}{1+c}.$$
Again using this fact and equation~(\ref{E4}) we obtain that
$$x_{n_{0}-2}x_{n_{0}-3}\ldots x_{n_{0}-k-2}=\frac{-1}{1+c+c^2}.$$
Continuing in this manner we get
$$x_{0}x_{-1}\ldots x_{-k}=\frac{-1}{\sum _{i=0}^{n_{0}}c^{i}}$$
or equivalently $(x_{-k}, x_{-k+1},\ldots x_{0})\in \mathscr{F}$.
This completes our inductive proof.\end{proof}
\section{Linearized Stability}
The first step to understand of the dynamics is to find the equilibria and find their stability type.
In this section we investigate the local asymptotic stability of
the equilibria of equation~(\ref{E4}). Simple calculations show
that origin is always an equilibrium for equation~(\ref{E4}) and
if $c>1$ then it has a second equilibrium
$\overline{x}=\sqrt[k+1]{c-1}$. The following theorem deals with
the local asymptotic stability of the equilibria. Note that
throughout the rest of this paper, the initial conditions are
assumed to be taken out of the forbidden set \(\mathscr{F}.\)

\begin{Theorem}\label{T4}Consider equation~(\ref{E4}). Then,
\begin{description}
    \item[\it{(a)}] origin is locally
    asymptotically stable for $c<1$, while it is unstable when $c>1$.
    \item[\it{(b)}] For the case of $c>1,$ the positive equilibrium
    $\overline{x}=\sqrt[k+1]{c-1}$ is stable.
\end{description}
\end{Theorem}
\begin{proof}  For part
(a), the linearized equation associated with equation~(\ref{E4})
about origin is $z_{n+1}-cz_{n-k}=0, \ n\geq 0$. Therefore, the
corresponding characteristic equation is $\lambda ^{k+1}-c=0$,
i.e., $\lambda =\sqrt[k+1]{c}$. Since the origin is locally
asymptotically stable if $c<1$ and is unstable if $c>1$.

(b) Simple calculations show that the linearized equation
associated with equation~(\ref{E4}) about the positive
equilibrium $\overline{x}=\sqrt[k+2]{c-1}$ is
\begin{eqnarray*}
z_{n+1}+\frac{c-1}{c}z_{n}+\frac{c-1}{c}z_{n-1}+\ldots +\frac{c-1}{c}z_{n-k+1}-\frac{1}{c}z_{n-k}=0, \ \ \ n=0,1,\ldots
\end{eqnarray*}
Therefore the corresponding characteristic equation is
\begin{eqnarray}\label{E5}
\lambda ^{k+1}+\frac{c-1}{c}\lambda ^{k}+\ldots
+\frac{c-1}{c}\lambda -\frac{1}{c}=0,
\end{eqnarray}
Some algebra show that equation~(\ref{E5}) is equivalent to
\begin{eqnarray}\label{E6}
\frac{(\lambda ^{k+1}-1)(\lambda -1/c)}{\lambda -1}=0, \ \ \
\lambda \neq 1,
\end{eqnarray}
Therefore, $\lambda =1/c$ is one of the roots of
equation~(\ref{E6}). Since $c<1,$ by local asymptotic stability
Theorem we conclude that the positive equilibrium
$\overline{x}=\sqrt[k+1]{c-1}$ is unstable. The proof is complete.
\end{proof}
\begin{Remark}\label{R1} The above theorem is sufficient to completely describe the local dynamics of the equilibrium for the case \(c<1.\) However, for the case \(c\leq 1,\) more investigation are needed.
For $c=1$ the characteristic equation about origin has modulus
equal to one, the linearization fails to analyze the asymptotic
stability of origin. Thus, it may have a complex dynamics near
this point. On the other hand consider equation~(\ref{E6}) and
assume that $c>1$. equation~(\ref{E6}) has a real root $\lambda
=1/c $ with modulus less than one but this equation has  $k+1$
roots in the following form
$$\lambda _{m}=\cos \left(\frac{2m\pi }{k+1}\right)+i\sin \left(\frac{2m\pi }{k+1}\right), \ \ \ m=0,1,\ldots ,k,$$
all with modulus equal to one. So, if $c>1$ then linearization
tells us nothing about the stability of the positive equilibrium.
In the next section we discuss these cases in details.
\end{Remark}

\section{Semiconjugate factorization}

The main purpose of this section is to analyze the local dynamics
near equilibria for the cases \(c\geq 1\) as well as of the
global asymptotic dynamics of the equation~(\ref{E4}). Although
the previous section, using linearization, showed that the origin
is locally asymptotically stable for $c<1$. Yet the global nature
of asymptotic stability can not be inferred from linearization.
Also, we saw that for $c> 1$ the linearized equation about
the positive equilibrium has several roots; all with modulus equal
to unity. As a result, more powerful methods are necessary in
order to analyze the dynamics of equation~(\ref{E4}). In this
section, we apply semiconjugacy analysis to examine the global
nature of equation~(\ref{E4}). The idea is to reduce the order of a higher order difference equation such that the
analysis of the reduced system is feasible. Then, this is helpful when this analysis can facilitate the understanding of the dynamics for the original difference equation.

Let $f(x_{1},x_{2},\ldots ,x_{k+1})=\frac{cx_{k+1}}{1+x_{1}x_{2}\ldots x_{k+1}}$. Consider
the vectorization of $f$ i.e.,
\begin{eqnarray*}
V_{f}(x_{1},x_{2},\ldots ,x_{k+1})=\left(\frac{cx_{k+1}}{1+x_{1}x_{2}\ldots x_{k+1}},x_{1}, \ldots
,x_{k}\right).
\end{eqnarray*}

Set $H(x_{1},x_{2},\ldots ,x_{k+1})=x_{1}x_{2}\ldots x_{k+1}$ and
note that
\begin{eqnarray*}
H(V_{f}(x_{1},x_{2},\ldots
,x_{k+1}))=\frac{cx_{1}x_{2}\ldots x_{k+1}}{1+x_{1}x_{2}\ldots
x_{k+1}}=\frac{cH(x_{1},x_{2},\ldots
,x_{k+1})}{1+H(x_{1},x_{2},\ldots ,x_{k+1})},
\end{eqnarray*}
So the map $H$ makes $V_{f}$ a $(D,H,\phi )$-semiconjugate map with
$D$ being the nonnegative orthant of $\Bbb{R}^{k+1}$, i.e., $D=[0,\infty )^{k+1}$, and the map
\begin{eqnarray*}
\phi (t)=\frac{ct}{1+t},
\end{eqnarray*}
serving as the factor map on $[0,\infty )$. Therefore,
equation~(\ref{E4}) is a semiconjugae factorization of the
well-known Ricatti difference equation $t_{n+1}=\phi
(t_{n})=\frac{ct_{n}}{1+t_{n}}$. In the sequence, we consider
three cases as follow:

\emph{Case I:} $0<c<1$; In this case we claim that
$$0\leq x_{n+1}\leq c^{\lceil n/k+1 \rceil } \max \{x_{-k},x_{-k+1},\ldots ,x_{0}\}, \ \ \ n\geq -k, $$
where $\lceil n\rceil $ is the greatest integer which is less than
or equal to $n$. To prove the above claim note that it is true for
all $-k\leq n \leq -1$. Assume that it is true for all integers
less than or equal to some integer $n(n>-1)$. Then
\begin{eqnarray*}
  x_{n+1} &=& \frac{cx_{n-k}}{1+x_{n}x_{n-1}\ldots x_{n-k}}\\
          &\leq & cx_{n-k} \\
          &\leq & c^{\lceil \frac{n-k}{k+1}\rceil+1 }\max \{x_{-k},x_{-k+1},\ldots
 ,x_{0}\}\\
          &=& c^{\lceil \frac{n+1}{k+1}\rceil }\max \{x_{-k},x_{-k+1},\ldots ,x_{0}\},
\end{eqnarray*}
and this completes our inductive proof. Therefore, in this case
every positive solution of equation~(\ref{E4}) converges
exponentially to origin.

\emph{Case\ II:} $c>1$; In this case the factor map $\phi $ has a
positive equilibrium $\overline{t}=c-1$. Recall that in a Ricatti
equation the positive equilibrium is globally asymptotically
stable (see \cite{KL}. P79). Hence,  $\overline{t}$ is globally
asymptotically stable equilibrium of $\phi $. Also, since the
sequence $\{\phi (t_{0})\}_{n=0}^{\infty },\ t_{0}\in (0,\infty )$
is bounded (this is evident by the global asymptotic stability of
$\overline{t}$), Theorem \ref{T2}\emph{(ii)} implies that the
sequence $\{V_{f}^{n}(x_{0})\}_{n=0}^{\infty }, \ x_{0}\in
(0,\infty )^{k+1}$ is bounded. Thereby, we conclude from Theorem
2\emph{(iv)} that the sequence
$\{V_{f}^{n}(x_{0})\}_{n=0}^{\infty }$ converges to the invariant
fiber $H_{\overline{t}}^{-1}$ since  the sequence $\{\phi
(t_{0})\}_{n=0}^{\infty }$ converges to $\overline{t}$. Now,
consider the fiber $H_{\overline{t}}^{-1}$. If
$(x_{1},x_{2},\ldots x_{k+1})\in H_{\overline{t}}^{-1}$ then some
simple algebra show that
$$V_{f}^{k+1}(x_{1},x_{2},\ldots ,x_{k+1})=(x_{1},x_{2},\ldots ,x_{k+1}),$$
Therefore, every member of $H_{\overline{t}}^{-1}$ is a
$(k+1)$-cycle of $V_{f}$; in other words, it is a periodic orbit of
equation~(\ref{E4}) of period $(k+1)$. Then, in this case every
solution of equation~(\ref{E4}) converges to a $(k+1)$-cycle.

\emph{Case\ III:} $c=1$; In this case origin is the unique
equilibrium of equation~(\ref{E4}) and $\phi $. Some calculations
show that
\begin{eqnarray*}
\phi '(0)=1, \ \ \ \phi ''(0)=-2,
\end{eqnarray*}
Therefore by Theorem \ref{T3} origin is semiasymptotic stable
equilibrium of $\phi $ form the right, {\rm i.e.,} origin attracts the
sequence $\{\phi ^{n}(t_{0})\}$ for all $t_{0}\in (0,\infty )$.
So, by an analysis precisely similar to that of used in the
previous case we conclude, by Theorem 2\emph{(iv)}, that the the
fiber $H_{0}^{-1}$ attracts every solution of equation~(\ref{E4})
since origin attracts every solution of the sequence $\{\phi
^{n}(t_{0})\}$. Now, assume that $(x_{1},x_{2},\ldots ,x_{k+1})\in
H_{0}^{-1}$. Then
\begin{eqnarray*}
x_{1}x_{2}\ldots x_{k+1}=0,
\end{eqnarray*}
and therefore, every solution of equation~(\ref{E4}) either
converges to origin or to a point in one of $(k+1)$ coordinate
planes in $\Bbb{R}^{k+1}$, i.e., the following set
\begin{eqnarray}\label{Cycle}
S=\{(x_{1},x_{2},\ldots ,x_{k+1})| \ \exists \ 1\leq i \leq k+1 \
\ni \ x_{i}=0 \ \}.
\end{eqnarray}
elements of the attracting set \(S\) (note that \(S=\cup S_i\),
where \(S_i\) are the coordinate planes) are \((k+1)\)-cycles of
equation~(\ref{E4}).
 To prove this consider $(x_{1}, \ldots ,x_{i},0,x_{i+1},\ldots  ,x_{k+1})\in S$. Then, it is easy to verify
that
\begin{eqnarray}
V_{f}^{k+1}(x_{1}, \ldots ,x_{i},0,x_{i+1},\ldots ,x_{k+1})=(x_{1}, \ldots ,x_{i},0,x_{i+1},\ldots
,x_{k+1}).
\end{eqnarray}
Consequently, every solution of equation~(\ref{E4}) either
converges to the origin or to a periodic orbit of period $(k+1),$
that is in a coordinate plane of $\Bbb{R}^{k+1}$. We now summarize the
above arguments into the following theorem that is one of the
main result of this paper.
\begin{Theorem}\label{T5}(Basic convergence Theorem)  Assume Equation (\ref{E4}) is given. Then,
\begin{description}
    \item[\it{(i)}] For $0<c<1,$ every positive orbit of equation~(\ref{E4}) converges exponentially to
     origin. This implies that the equation has no other periodic orbit except the trivial fixed point ({\rm i.e.,} origin).
    \item[\it{(ii)}] When $c\geq 1,$ any cycle of the equation~(\ref{E4}) is a \((k+1)\)-cycle and the set of all \((k+1)\)-cycles is given by equation~(\ref{Cycle}). Furthermore 
        \begin{enumerate}
          \item For \(c=1,\) every positive orbit of equation~(\ref{E4})
    either converges to origin or to a $(k+1)$-cycle.
          \item When $c>1,$ every positive solution of equation~(\ref{E4}) converges to a $(k+1)$-cycle.
        \end{enumerate}
\end{description}
\end{Theorem}

\end{document}